\documentclass[11pt, letterpaper, oneside]{amsart}
\usepackage{tikz}
\usepackage{amsrefs}
\usepackage{amsthm}
\usepackage{mathtools}
\usepackage{genyoungtabtikz}
\usepackage{float}

\usetikzlibrary{decorations.pathreplacing,calligraphy}
\usetikzlibrary { decorations.pathmorphing, decorations.pathreplacing, decorations.shapes, }
\usepackage{caption}

\DeclareCaptionFormat{myformat}{\vspace{10pt}#1#2#3}

\newtheorem{thm}{Theorem}[section]
\newtheorem{cor}[thm]{Corollary}

\newtheorem{prop}[thm]{Proposition}
\theoremstyle{definition}
\newtheorem{definition}[thm]{Definition}
\newtheorem{exm}[thm]{Example}

\numberwithin{equation}{section}

\begin{document}

\title[Partitions with Fixed Hooks]{Partitions with fixed points in the sequence of first-column hook lengths}
\author{Philip Cuthbertson}
\address{Department of Mathematical Sciences\\
Michigan Technological University\\ Houghton, MI 49931} \email{pecuthbe@mtu.edu}

\author{David J.\ Hemmer}
\address{Department of Mathematical Sciences\\
Michigan Technological University\\ Houghton, MI 49931} \email{djhemmer@mtu.edu}

\author{Brian Hopkins}
\address{Department of Mathematics and Statistics\\Saint Peter's University\\Jersey City, NJ 07306} \email{bhopkins@saintpeters.edu}
\date{October, 2023}

\author{William Keith}
\address{Department of Mathematical Sciences\\
Michigan Technological University\\ Houghton, MI 49931} \email{wjkeith@mtu.edu}
\date{December, 2023}
\begin{abstract}

Recently, Blecher and Knopfmacher applied the notion of fixed points to integer partitions.  This has already been generalized and refined in various ways such as $h$-fixed points for an integer parameter $h$ by Hopkins and Sellers.  Here, we consider the sequence of first column hook lengths in the Young diagram of a partition and corresponding \textit{fixed hooks}.  We enumerate these, using both generating function and combinatorial proofs, and find that they match occurrences of part sizes equal to their multiplicity.  We establish connections to work of Andrews and Merca on truncations of the pentagonal number theorem and classes of partitions partially characterized by certain minimal excluded parts (mex).
\end{abstract}

\maketitle

\section{Introduction}
Let $n$ be a positive integer. Recall that a \textit{partition} of $n$, denoted $\lambda \vdash n$, is a sequence $\lambda=(\lambda_1, \lambda_2, \ldots, \lambda_t)$ with $\lambda_i \geq \lambda_{i+1}>0$ and $\sum_{i=1}^t\lambda_i=n$.

Recently, Blecher and Knopfmacher 
\cite{BlecherFixedPointsRamanujan} 
expanded the notion of permutation fixed points to integer partitions.

\begin{definition}
A partition $\lambda$ has a \emph{fixed point} if there exists $i$ with $\lambda_i=i$.    
\end{definition}
Blecher and Knopfmacher studied fixed points both for partitions as defined above (where a partition can have at most one fixed point) and for partitions where the entries are written in nondecreasing order (where there can be multiple fixed points). We will focus on the first case.

Hopkins and Sellers 
\cite{hopkins2023blecher}
defined a generalization of partition fixed points.  
\begin{definition}
\label{def:partitionkfixedpoint}
Given an integer $h$, a partition $\lambda$ has an $h$-\emph{fixed point} if there exists $i$ with $\lambda_i=i+h$.  
\end{definition}
See also Hopkins's combinatorial refinement which distinguishes partition fixed points by which part is fixed
\cite{hopkins23bk2}.

Given $\lambda \vdash n$, there is an important nonincreasing sequence of integers associated to $\lambda$, namely the \textit{first-column hook lengths}, also known as a \emph{sequence of} $\beta$-\emph{numbers} \cite{jamescombinatorialresults}
, which we describe now. These play a critical role, for example, in the representation theory of symmetric groups.

Let $\lambda=(\lambda_1, \lambda_2, \ldots, \lambda_t)$ be a partition of $n$ with corresponding Young diagram $[\lambda]$. Let $\lambda'$ denote the conjugate or transpose partition, defined by $\lambda^\prime_i = \#\{\lambda_j \leq i \}$. For a box $(i,j) \in [\lambda]$, recall that the $i,j$-\emph{hook length} is defined as:
$$h_{i,j}(\lambda) =\lambda_i + \lambda'_j-i-j+1.$$ 
\begin{definition}
The \emph{sequence of first column hook lengths} is given by $\{h_{1,1}(\lambda), h_{2,1}(\lambda), \ldots, h_{t,1}(\lambda)\}$. 

\end{definition}
Notice that the partition $\lambda$ can easily be recovered from the corresponding sequence of first column hook lengths. It turns out that studying fixed points in this sequence leads to very interesting combinatorics!

We define a version of fixed hooks following Definition \ref{def:partitionkfixedpoint} for partitions:
\begin{definition} For $h \in \mathbb{Z}$, say that $\lambda$ has an \emph{$h$-fixed hook} if there is some $i \geq 1$ such that $h_{i,1}(\lambda)=i+h$. 
\end{definition}

See Figure \ref{S1example} for an example of the various definitions.  It shows the Young diagram of the partition $(2,2,1)$ with hook lengths in each box.  The first column hook lengths are $\{4,2,1\}$.  Note that the partition has a fixed hook $h_{2,1}$, a 3-fixed hook $h_{1,1}$, and a $-2$-fixed hook $h_{3,1}$.

\begin{figure}
\young(42,21,1)
\caption{The Young diagram of $(2,2,1)$ with hook lengths.}
\label{S1example}
\end{figure}

The sequence of first-column hook lengths is strictly decreasing, so if $\lambda$ has an $h$-fixed hook for a given $h$, then it is unique. We will be interested in counting partitions with $h$-fixed hooks for various $h$ both positive and negative.

In the sequel we make heavy use of the following standard notation: 
\begin{gather*}
(a;q)_\infty = \prod_{k=0}^\infty (1-a q^k),  \; (a;q)_n = \frac{(aq^n;q)_\infty}{(a;q)_\infty} , \\   
(q)_x = (q;q)_x, \; \left[ {A \atop B} \right]_q = \frac{(q)_A}{(q)_B(q)_{A-B}}.
\end{gather*}
These have the following combinatorial interpretations (see Andrews \cite{AndrewsBook}): $\left[ {{M+N} \atop N} \right]_q$ is the generating function for partitions with at most $N$ parts and largest part at most $M$, and $\frac{1}{\prod_{i \in S} (1-q^i)}$ is the generating function for partitions with parts in $S$, e.g., $\frac{1}{(q)_n}$ is the generating function for partitions with parts at most $n$.

\section{Fixed hooks}

In this section we consider partitions $\lambda$ with 0-fixed hooks $h_{i,1}(\lambda)=i$, which when clear from context we will simply call fixed hooks.

The main result of this section is:

\begin{thm}
  \label{thm: main counting result}

The number $f(n)$ of partitions $\lambda \vdash n$ having a fixed hook is equal to the sum over all partitions of n of the number of distinct parts $i$ of multiplicity $i$.

\end{thm}

The sequence $\{f(n)\}$  in Theorem \ref{thm: main counting result} is A276428 in the Online Encyclopedia of Integer Sequences \cite{oeis}.

We will give a short generating function proof, and then a more detailed combinatorial proof which illuminates additional features of the correspondence.  Readers may note an interesting feature of the bijective proof is that a partition may have at most one fixed hook, whereas several part sizes may be equal to their multiplicity, as in $(3,2,2,1)$. So Theorem \ref{thm: main counting result} is not an equality between two sets of partitions.

\begin{proof}[Generating function proof]
    Let $\lambda$ have $k$ parts.  A fixed hook occurs at part $\lambda_{k-j}$ if $\lambda_{k-j} +j = k-j$, so $\lambda_{k-j} = k-2j$.  There are $j$ further parts of sizes between 1 and $k-2j$, generated by $q^j \left[ {k-j-1} \atop j \right]_q$.  There are $k-j-1$ prior parts of size at least $k-2j$, generated by $q^{(k-j-1)(k-2j)} \frac{1}{(q)_{k-j-1}}$.  Summing over all $k$ and $j$, the generating function $\sum  f(n)q^n$ for fixed hooks is

$$\sum_{k=1}^\infty \sum_{j=0}^\infty q^{k^2 - 3kj + 2j^2 +j} \frac{1}{(q)_{k-2j-1} (q)_j}.$$

Setting $T = k - 2j - 1$, this sum becomes

$$\sum_{T=0}^\infty \sum_{j=0}^\infty \frac{q^{(T+1)^2 + (T+2)j}}{(q)_T (q)_j} = \sum_{T=0}^\infty \frac{q^{(T+1)^2}}{(q)_T} \sum_{j=0}^\infty \frac{q^{(T+2)j}}{(q)_j}.$$

Using the identity $\sum_{j=0}^\infty \frac{q^{(T+2)j}}{(q)_j} = \frac{1}{(q^{T+2};q)_\infty}$ (combinatorially, each side describes partitions with parts of size at least $T+2$), the generating function becomes

$$\sum_{T=0}^\infty \frac{q^{(T+1)^2}}{(q)_T(q^{T+2};q)_\infty} = \sum_{T=0}^\infty \frac{q^{(T+1)^2} (1-q^{T+1})}{(q)_\infty}.$$

This is the generating function counting the number of times in all partitions of $n$ that part size $i$ appears with exactly multiplicity $i$.
\end{proof}

For a bijective proof of Theorem \ref{thm: main counting result} we use an interesting bijection found on the Mathematics Stack Exchange 
\cite{OverflowBij}.
Let $P_a$ denote the set of all partitions into at most $a$ parts. Let $R_{a,b}$ be all partitions that fit into a rectangle with $a$ rows and $b$ columns.
\begin{prop}
  \label{prop: Splutterwitbijection}  There is a bijection:

  $$F_{a,b}: P_a \times P_b \rightarrow P_{a+b} \times R_{b,a}$$
  which preserves the total weight of the two partitions on each side.
\end{prop}

Note that this bijection realizes the identity:
\begin{equation*}\label{eq: qbinomialidentity}
  \frac{1}{(q)_a} \cdot \frac{1}{(q)_b} = \frac{1}{(q)_{a+b}}  
  \left[ \begin{array}{c}
    a+b \\
a
\end{array} \right]_q.
\end{equation*}
\begin{proof}
Given a partition $\lambda=(\lambda_1, \lambda_2, \ldots, \lambda_t)$ and a positive integer $R$ we describe how to ``insert" $R$ into $\lambda$. If $R<\lambda_t$ then simply place it at the end of $\lambda$. If not then let it ``slide past" $\lambda_t$ while subtracting one from $R$. Now compare $R-1$ to $\lambda_{t-1}$ and repeat until you reach the smallest $s$ with $R-s \leq \lambda_{t-s}$, so there are no further slides. Record the new partition $\tilde{\lambda}=(\lambda_1, \lambda_2, \ldots, \lambda_{t-s},R-s,\lambda_{t-s+1}, \ldots, \lambda_t)$ together with the integer $s \leq t$ counting the number of slides.

To define $F_{a,b}(\lambda, \mu)$, successively insert $\mu_1, \mu_2, \ldots, \mu_b$ into $\lambda$ as above. The resulting partition has at most $a+b$ parts. Keep track of the sequence of $s$ values, which will have at most $b$ parts, each at most $a$. It is easy to visualize this writing $\mu$ atop $\lambda$ in the French notation as in Figure \ref{fig:insertion}, which demonstrates for $a=4, b=3$, that
\begin{equation*}
    F_{4,3}((7,5,3,2),(9,7,5))=((7,6,5,5,3,3,2),(3,2,2)).
\end{equation*}

  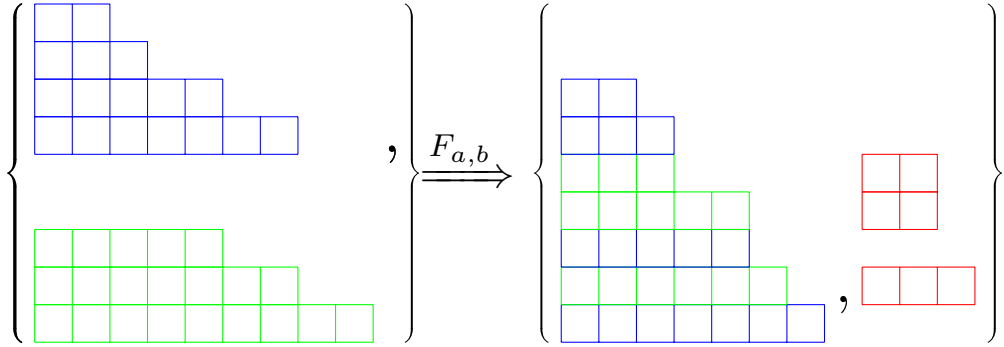
\begin{figure}
\centering

   \begin{tikzpicture}[scale=0.5]

   \draw [black,
    decorate,
    decoration = {calligraphic brace,
        raise=5pt,
        amplitude=5pt,
        aspect=0.5},line width=0.75pt] (9.5,9) --  (9.5,0);
        
          \draw [black,
    decorate,
    decoration = {calligraphic brace,
        raise=5pt,
        amplitude=5pt,
        aspect=0.5},line width=0.75pt] (0,0) --  (0,9);

        \draw (9.5,5.0) node[scale=2.0, black] {,};

         \draw [black,
    decorate,
    decoration = {calligraphic brace,
        raise=5pt,
        amplitude=5pt,
        aspect=0.5},line width=0.75pt] (25,9) --  (25,0);

        \draw [black,
    decorate,
    decoration = {calligraphic brace,
        raise=5pt,
        amplitude=5pt,
        aspect=0.5},line width=0.75pt] (14,0) --  (14,9);

          \draw [black,
    decorate,
    decoration = {calligraphic brace,
        raise=5pt,
        amplitude=5pt,
        aspect=0.5},line width=1.25pt] (0,0) --  (0,9);

        \draw (9.5,5.0) node[scale=2.0, black] {,};
          \draw (21.5,1.0) node[scale=2.0, black] {,};

\draw [blue](0,5) grid (7,6);
\draw [blue](0,6) grid (5,7);
\draw [blue](0,7) grid (3,8);
\draw [blue](0,8) grid (2,9);

\draw [green](0,0) grid (9,1);
\draw [green](0,1) grid (7,2);
\draw [green](0,2) grid (5,3);

\draw [blue](14,0) grid (21,1);
\draw [green](14,1) grid (20,2);
\draw [red](22,3) grid (24,4);
\draw [blue](14,2) grid (19,3);
\draw [green](14,3) grid (19,4);

\draw [green](14,4) grid (17,5);
\draw [blue](14,5) grid (17,6);
\draw [blue](14,6) grid (16,7);
\draw [red](22,1) grid (25,2);
\draw [red](22,4) grid (24,5);
 \draw (11.5,4.5) node[scale=1.5, black] {$\xRightarrow{F_{a,b}}$};

    \end{tikzpicture}

\caption{The bijection $F_{a,b}$ for $a=4$, $b=3$.}
  \label{fig:insertion}

\end{figure}
Notice that given the pair of partitions on the right side of Figure \ref{fig:insertion}, \textit{together with the values of $a$ and $b$}, it is easy to recover $\lambda$ and $\mu$, so we have a bijection as desired. For example, the second part 2 in $(3,2,2)$ tells us the last green part slid two steps into $\lambda$. It is easy to reverse the two slides and recover the part size 5 in $\mu$. Notice if $b$ were four with the same image pair we would consider $(3,2,2,0)$ and the inverse map is not the same.
\end{proof}

We now give a bijective proof of Theorem \ref{thm: main counting result}.
\begin{proof}[Bijective proof]
  The input to the bijective proof is an ordered pair $(\lambda, i)$ so that $\lambda \vdash n$ and $i$ appears with multiplicity $i$ in $\lambda$. The output $\mathcal{B}(\lambda, i)$ is a partition $\mu \vdash n$ with a fixed hook. The bijection is such that if  $\mu=\mathcal{B}(\lambda,i)$  has $h_{1,r}(\mu)=r$, then $\mu_r=i$. Thus we are really giving a sequence of bijections, one for each $i$, between partitions with parts $i$ of multiplicity $i$ and partitions with fixed hooks corresponding to a part of length $i$.

  Given $(\lambda,i)$, assume that the $i$ identical parts $i$ start in row $k$ for some $k \geq 1$. Thus $\lambda_{k-1}>i$, $\lambda_{k+i}<i$ and $\lambda_{k}=\lambda_{k+1}= \cdots = \lambda_{k+i-1}=i$.

  The Young diagram of $\lambda$ is shown in Figure \ref{fig: lambda with ii}.

  \begin{figure}

\centering

   \begin{tikzpicture}[scale=0.5]

           \fill[cyan, opacity=0.2] (0,0) rectangle (3,3);
            \fill[blue, opacity=0.2] (0,3) rectangle (3,6);
      
        \draw (-4,3) node[scale=1.0, black] {$[\lambda]=$};
         \draw (5,4.8) node[scale=1, blue] {$\tau$};
          \draw (.5,-1) node[scale=1, red] {$\epsilon$};
          \draw (6,1) node[scale=1, blue] {$\tau \in P_{k-1}$};
            \draw (6,-1) node[scale=1, red] {$\epsilon' \in P_{i-1}$};

        \draw [thick] (0,0) -- (0,6);
        \draw [thick] (3,6) -- (3,0);
        \draw [thick] (3,0) -- (0,0);
        \draw[thick] (0,3)--(0,6);
        \draw[thick] (0,6)--(4,6);
        \draw[thick] (3,3)--(4,3);
        \draw[thick] (4,3)--(4,6);
        \draw[blue][thick] (4,6)--(7,6);
        \draw[red][thick] (0,0)--(0,-3);
        \draw[blue][thick] (4,3)--(4,6);

    \draw [red,
    decorate,
    decoration = {calligraphic brace,
        raise=5pt,
        amplitude=5pt,
        aspect=0.5},line width=1.25pt] (0,3) --  (0,6)
node[pos=0.5,left=10pt,black, scale=1]{$k-1$};

  \draw [red,
    decorate,
    decoration = {calligraphic brace,
        raise=5pt,
        amplitude=5pt,
        aspect=0.5},line width=1.25pt] (0,0) --  (0,3)
node[pos=0.5,left=10pt,black, scale=1]{$i$};

   \draw [red,
    decorate,
    decoration = {calligraphic brace,
        raise=5pt,
        amplitude=5pt,
        aspect=0.5},line width=1.25pt] (0,6) --  (3,6)
node[pos=0.5,above=10pt,black, scale=1]{$i$};

       \draw[decorate,decoration={coil,aspect=0}, blue, thick]   (7,6)--(4,3);
       \draw[decorate,decoration={coil,aspect=0}, red, thick]   (0,-3)--(2,0);

       \foreach \i in {0,...,2}
         \draw (3.5,5.5-\i) node[scale=1][blue] {X};
               \foreach \i in {0,...,2}
             \draw (.5+\i,.5) node[scale=1][blue] {X};

    \end{tikzpicture}

\caption{Partition $\lambda$ with $i$ copies of $i$ starting in row $k$.}
  \label{fig: lambda with ii}
  \end{figure}

Next suppose that $\mu \vdash n$ has $h_{k+i-1}(\mu)=k+i-1$ and $\mu_{k+i-1}=i$. These two conditions immediately force $\mu$ to have exactly $i+2k-2$ nonzero parts, so its Young diagram is as shown in Figure \ref{fig:muwithHL}.

  \begin{figure}

\centering

   \begin{tikzpicture}[scale=0.5]

        \fill[blue, opacity=0.2] (0,3) rectangle (3,6);
        \fill[cyan, opacity=0.2] (0,0) rectangle (3,3);
            
       \draw (-4,3) node[scale=1.0, black] {$[\mu]=$};

         \draw [thick] (0,0) -- (0,6);
         \draw [thick] (0,1) -- (3,1);
         \draw [thick] (3,6) -- (3,0);
         \draw [thick] (3,0) -- (0,0);
        \draw[thick] (0,3)--(0,6);
        \draw[thick] (0,6)--(3,6);
        \draw[blue][thick] (3,6)--(7,6);
        \draw[blue][thick] (0,0)--(0,-3);
        \draw[blue][thick] (0,-3)--(1,-3);
        \draw[blue][thick] (1,0)--(1,-3);
        \draw[red][thick] (1,-3)--(3,-3);
      \draw[red][thick] (3,0)--(3,-3);

       \foreach \i in {0,...,2}
         \draw (.5,-.5-\i) node[scale=1][blue] {X};
            \foreach \i in {0,...,2}
         \draw (.5+\i,.5) node[scale=1][blue] {X};

    \draw [red,
    decorate,
    decoration = {calligraphic brace,
        raise=5pt,
        amplitude=5pt,
        aspect=0.5},line width=1.25pt] (0,3) --  (0,6)
node[pos=0.5,left=10pt,black, scale=1]{$k-1$};

  \draw [red,
    decorate,
    decoration = {calligraphic brace,
        raise=5pt,
        amplitude=5pt,
        aspect=0.5},line width=1.25pt] (0,0) --  (0,1)
node[pos=0.5,left=10pt,black, scale=1]{$1$};

  \draw [red,
    decorate,
    decoration = {calligraphic brace,
        raise=5pt,
        amplitude=5pt,
        aspect=0.5},line width=1.25pt] (0,1) --  (0,3)
node[pos=0.5,left=10pt,black, scale=1]{$i-1$};

   \draw [red,
    decorate,
    decoration = {calligraphic brace,
        raise=5pt,
        amplitude=5pt,
        aspect=0.5},line width=1.25pt] (0,6) --  (3,6)
node[pos=0.5,above=10pt,black, scale=1]{$i$};

       \draw[decorate,decoration={coil,aspect=0}, blue, thick]   (7,6)--(3,1);
        \draw (4,4) node[scale=1, blue] {$\gamma$};

      \draw (2,-1.5) node[scale=1, red] {$\rho$};

       \draw [blue,
    decorate,
    decoration = {calligraphic brace,
        raise=5pt,
        amplitude=5pt,
        aspect=0.5},line width=1.25pt] (0,-3) --  (0,0)
node[pos=0.5,left=10pt,black, scale=1]{$k-1$};
\draw (6,1) node[scale=1, blue] {$\gamma \in P_{i+k-2}$};
            \draw (6,-1) node[scale=1, red] {$\rho \in R_{k-1, i-1}$};

    \end{tikzpicture}

\caption{ $\mu $ with $h_{k+i-1}(\mu)=k+i-1$ and $\mu_{k+i-1}=i$.}
  \label{fig:muwithHL}

\end{figure}

In Figure \ref{fig: lambda with ii}, $\tau$ is an arbitrary partition in $P_{k-1}$, and $\epsilon$ is an arbitrary partition with all parts at most $i-1$ (equivalently, $\epsilon ' \in P_{i-1}$). In Figure \ref{fig:muwithHL}, $\gamma$ is an arbitrary partition in $P_{k+i-2}$ and $\rho$ is an arbitrary partition in $R_{k-1,i-1}$.

Now the bijection is easy to define. Define $\mu=\mathcal{B}(\lambda, i)$ by starting with $\lambda$. First move the $k-1$ blue boxes down from $\{(1,i+1), (2,i+1), \ldots, (k-1,i+1)\}$ to $\{(i+k, 1), (i+k+1,1), \ldots, (i+2k-2,1)\}$. Then construct $\mu$ by setting:
$$(\gamma,\rho):=F_{k-1, i-1}(\tau, \epsilon').$$ 

The inverse is easily obtained by first performing the inverse of $F$ and then moving the blue boxes up. 
\end{proof}

Here is an example of the bijection.

\begin{exm}
  Let $\lambda=(16,12,8,8,7,5,5,5,5,5,4,4,3,3,3,2)\vdash 95$ and $i=5$. Then $k=6$, $\tau=(10,6,2,2,1)$ and ${\color{red} \epsilon'=(6,6,5,2)}$. We calculate that:

  $$F_{5,4}(\tau, \epsilon')=\{(10, 6, {\color{red} 3,3,2}, 2, 2, {\color{red}{1}}, 1), ({\color{red}3, 3, 3, 1})\}$$

The red encodes how $\epsilon '$ ``sinks" into $\tau$ under the map of Proposition \ref{prop: Splutterwitbijection} as in Figure \ref{fig:insertion}. Thus  $\mu=\mathcal{B}(\lambda, 5)=(15,11,8,8,7,7,7,6,6,5,4,4,4,2,1) \vdash 95.$ Notice that $h_{10,1}(\mu)=10$ as desired with $\mu_{10}=5$ as guaranteed by Proposition \ref{prop: fixedpointresultbyi} below.

\end{exm}

The following strengthening of Theorem \ref{thm: main counting result} is clear from the bijection:

\begin{prop}
    \label{prop: fixedpointresultbyi}
The number of partitions $\lambda \vdash n$ which contain the part $i$ with multiplicity $i$ is the same as the number of partitions $\mu \vdash n$ for which there is a fixed hook $h_{s,1}=s$ with $\mu_s=i$.
\end{prop}

Table \ref{tableProp33} illustrates Proposition \ref{prop: fixedpointresultbyi} and the corresponding bijection for $n=9$. On the left we see 12 occurrences of $i$ with multiplicity $i$, seven for $i=1$, four for $i=2$ and one for $i=3$. Notice that the partition $(4,2,2,1)$ appears twice. On the right we have the twelve partitions with fixed hooks, where the part $\lambda_i$ with $h_{i,1}(\lambda)=i$ is highlighted. Notice that the highlighted part sizes occur with multiplicities seven, four, and one as predicted.

\begin{table}[h]
    \centering
    \caption{Proposition \ref{prop: fixedpointresultbyi} for $n=9$.}
    \begin{tabular}{|l|l|}
        \hline
        $\lambda \vdash 9$ with parts $i$ of multiplicity $i$ & $\mathcal{B}(\lambda,i) \vdash 9$ with fixed hooks \\
        \hline
        (8,\textcolor{red}{1}) &  (7,\textcolor{red}{1},1) \\
        (6,2,\textcolor{red}{1}) &(5,1,\textcolor{red}{1},1,1)\\
        (5,3,\textcolor{red}{1}) & (4,2,\textcolor{red}{1},1,1)  \\
        (5,\textcolor{red}{2,2}) & (4,2,\textcolor{red}{2},1) \\
        (4,4,\textcolor{red}{1}) &(3,3,\textcolor{red}{1},1,1) \\
        (4,\textcolor{red}{2,2},1) &  (3,2,\textcolor{red}{2},2)  \\
        (4,2,2,\textcolor{red}{1})&(3,1,1,\textcolor{red}{1},1,1,1)\\
        (\textcolor{red}{3,3,3}) & (3,3,\textcolor{red}{3})  \\
        (3,3,2,\textcolor{red}{1}) &  (2,2,1,\textcolor{red}{1},1,1,1)\\
        (3,\textcolor{red}{2,2},1,1) &(3,3,\textcolor{red}{2},1)\\
        (2,2,2,2,\textcolor{red}{1}) & (1,1,1,1,\textcolor{red}{1},1,1,1,1) \\
        (\textcolor{red}{2,2},1,1,1,1,1)& (7,\textcolor{red}{2})  \\
     
        \hline
    \end{tabular}
    \label{tableProp33}
\end{table}

\section{Fixing part sizes}

The final result of the previous section, illustrated in Table \ref{tableProp33}, suggests that it would be interesting to specify the part sizes corresponding to fixed hooks.  In this section we do so, and find a close connection with an investigation by Andrews and Merca \cite{AndrewsMercaTruncatedPentagonal} on the truncated pentagonal number theorem.  We give brief summary of this result.

The generating function for the partition number can be written $$\sum_{n=0}^\infty p(n) q^n = \frac{1}{(q)_\infty} = \frac{1}{1-q-q^2+q^5+q^7-q^{12}-q^{15}+\dots},$$ where the latter denominator is described by the \emph{pentagonal numbers}, $$(q)_\infty = \sum_{n \in \mathbb{Z}} (-1)^n q^{\frac{n}{2}(3n-1)}.$$  The relation $$\frac{(q)_\infty}{(q)_\infty} = 1$$ now yields the \emph{pentagonal number recurrence} $p(n) = 0$ for $n<0$, $p(0)=1$, and for $n>0$, $$p(n)=p(n-1) + p(n-2) - p(n-5) - p(n-7) + p(n-12) + \cdots .$$

The difference $p(n)-p(n-1)$ can be seen to be the number of partitions of $n$ without 1s, by noting that a 1 can be appended to all partitions of $n-1$.  It is less obvious what truncations of the above recurrence, such as $$p(n)-p(n-1)-p(n-2)+p(n-5),$$ might describe, if anything.  Andrews and Merca \cite{AndrewsMercaTruncatedPentagonal} showed that the truncation after $2k$ terms is $(-1)^{k+1} M_k (n)$, where $M_k(n)$ is the number of partitions of $n$ in which the smallest part size not appearing (known as the \emph{minimal excludant} or \emph{mex} of the partition) is $k$, and the number of parts larger than $k$ is greater than the number of parts smaller than $k$.  They showed that this quantity has generating function $${\mathcal{M}}_k = \sum_{n=0}^\infty M_k(n) q^n = \sum_{n=k}^\infty \frac{q^{\binom{k}{2}+(k+1)n}}{(q)_n} \left[ {{n-1} \atop {k-1}} \right]_q.$$

 Recall that an $h$-\emph{fixed hook} is a hook $h_{s,1}(\lambda)=s+h$. A useful tool for the results in this section will be the generating function for the number of partitions of $n$ with an $h$-fixed hook where the part indexed in the place of the hook is size $k$.

\begin{thm}\label{hfixsizek} The generating function for the number of partitions of $n$ with an $h$-fixed hook arising from a part of size $k$ is \begin{multline*}\sum_{s=k-h}^\infty q^{(k+1)(s-1)+h+1} \left[{{s+h-1} \atop {k-1}} \right]_q \frac{1}{(q)_{s-1}} \\ = q^{h+1-\binom{k}{2}} \sum_{s=k-h}^\infty q^{(k+1)(s-1)+\binom{k}{2}} \left[ {{s+h-1} \atop {k-1}} \right]_q \frac{1}{(q)_{s-1}}.\end{multline*}
\end{thm}

\begin{proof}The proof follows from considering the contributions made by the various portions of the Young diagram of such a partition, shown in Figure \ref{fig:lambadforGF}.

 \begin{figure}[H]

\centering

   \begin{tikzpicture}[scale=0.5]

          \fill[cyan, opacity=0.2] (0,1) rectangle (4,4);
           \draw (-5,2) node[scale=1.0, black] {$[\lambda]=$};
               \draw (4.5,2.5) node[scale=1, blue] {$\gamma$};
                \draw (2,-1.5) node[scale=1, red] {$\rho$};

        \draw [thick] (0,0) -- (0,4);
        \draw [thick] (0,4) -- (4,4);
        \draw [thick] (4,4) -- (4,0);
        \draw [thick] (4,0) -- (0,0);
        \draw[thick] (0,0)--(0,-3);
        \draw[thick] (0,-3)--(1,-3);
        \draw[thick] (1,0)--(1,-3);   
        \draw[red][thick] (1,-3)--(4,-3);
      \draw[red][thick] (4,0)--(4,-3);
        \draw[blue][thick] (4,4)--(7,4);

       \foreach \i in {0,...,2}
         \draw (.5,-.5-\i) node[scale=1]{X};
           \foreach \i in {0,...,3}
         \draw (.5+\i,.5) node[scale=1]{X};

       \draw[decorate,decoration={coil,aspect=0}, blue, thick]   (7,4)--(4,1);

    \draw [red,
    decorate,
    decoration = {calligraphic brace,
        raise=5pt,
        amplitude=5pt,
        aspect=0.5},line width=1.25pt] (0,1) --  (0,4)
node[pos=0.5,left=10pt,red, scale=1]{$s-1$};

   \draw [red,
    decorate,
    decoration = {calligraphic brace,
        raise=5pt,
        amplitude=5pt,
        aspect=0.5},line width=1.25pt] (0,4) --  (4,4)
node[pos=0.5,above=10pt,red, scale=1]{$k$};

       \draw [blue,
    decorate,
    decoration = {calligraphic brace,
        raise=5pt,
        amplitude=5pt,
        aspect=0.5},line width=1.25pt] (0,-3) --  (0,0)
node[pos=0.5,left=10pt,red, scale=1]{$s+h-k$};
\draw (11,3) node[scale=1, blue] {$\gamma \in P_{s-1}$};
            \draw (11,1) node[scale=1, red] {$\rho \in R_{s+h-k, k-1}$};

    \end{tikzpicture}

\caption{$\lambda$ with an $h$-fixed hook $h_{1,s}=s+h$ at part $\lambda_s=k$.}
  \label{fig:lambadforGF}

\end{figure}
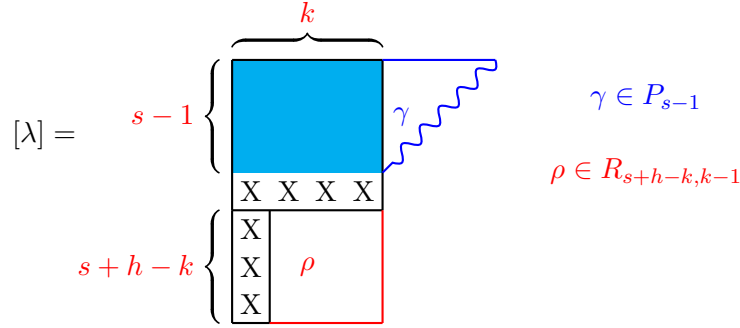

Let $s$ index the position where the hook is to be found.

There are $s-1$ parts of size at least $k$, generated by $q^{k(s-1)} / (q)_{s-1}$.

There is a hook of size $s+h$, generated by $q^{s+h}$.

The leg of the hook is of length $s+h-k$, meaning that the $\rho$ portion of the partition is contributed by the $q$-binomial coefficient $\left[ {{s+h-1} \atop {k-1}} \right]_q$.
\end{proof}

We may now state the main theorems of this section. The first considers fixed hooks corresponding to parts of size one.

\begin{thm}\label{sizeone} The number of partitions of $n$ with an $h$-fixed hook arising from a part of size 1, for $h \geq -1$, equals the number of times in all partitions of $n$ that 1 appears exactly $h+1$ times, with the exception of $n=0$, $h=-1$.
\end{thm}

\begin{proof} Set $k=1$ in Theorem \ref{hfixsizek}.  The generating function becomes $$\sum_{s=1-h}^\infty q^{2(s-1)+h+1} \left[{{s+h-1} \atop {0}} \right]_q \frac{1}{(q)_{s-1}} = q^{h+1} \sum_{s=1-h}^\infty  \frac{q^{2(s-1)}}{(q)_{s-1}} = \frac{q^{h+1}}{(q^2;q)_\infty}$$

\noindent where the last equality holds as long as $h>-1$, and this is the generating function for a partition with exactly $h+1$ parts of size 1, and all other parts of size at least 2.  When $h = -1$ the sum becomes $\frac{1}{(q^2;q)_\infty} - 1$, the missing term being exactly the exception noted (for $n=0$, the part size 1 appears 0 times once but there is no $-1$-fixed hook arising from a part of size 1).
\end{proof}

A similar proof gives a slightly different interpretation that extends to $h$-fixed hooks for any $h$.

\begin{thm}\label{sizeonesecond} The number of partitions of $n$ with an $h$-fixed hook arising from a part of size 1, for $h \in \mathbb{Z}$, equals the number of times in all partitions of $n-h$ that have length at least $1-h$ where $1$ appears exactly one time.
\end{thm}

\begin{proof} Set $k=1$ in Theorem \ref{hfixsizek}.  The generating function becomes $$\sum_{s=1-h}^\infty q^{2(s-1)+h+1} \left[{{s+h-1} \atop {0}} \right]_q \frac{1}{(q)_{s-1}} = q^{h+1} \sum_{s=1-h}^\infty  \frac{q^{2(s-1)}}{(q)_{s-1}}$$
and by letting $m=s+1$ we get 
$$q^{h+1} \sum_{m=-h}^\infty  \frac{q^{2m}}{(q)_{m}} = q^{h+1} \left( \sum_{m=0}^\infty  \frac{q^{2m}}{(q)_{m}} - \sum_{m=0}^{-h-1}  \frac{q^{2m}}{(q)_{m}}\right),$$
noting that the second sum is equal to $0$ for $h \geq 0$. Finally one gets
$$q^{h+1} \left( \frac{1}{(q^2;q)_\infty} - \sum_{m=0}^{-h-1}  \frac{q^{2m}}{(q)_{m}}\right).$$

\noindent This is precisely $q^h$ times the generating function for partitions that have one part of size 1, all other parts of size at least 2, and at least $1-h$ parts as desired.
\end{proof}

Finally we have the following theorem giving a set of partitions equinumerous with the number of $h$-fixed hooks in partitions of $n$ arising from parts of size 2 or greater.

\begin{thm}\label{minusonefixed} The number of times in all partitions of $n$ that an $h$-fixed hook arises from a part of size $k$ equals the number of partitions of $n + \binom{k}{2} - (h + 1)$ with mex $k$ where ($h+1+{}$the number of parts larger than $k$) is greater than the number of parts less than $k$.
\end{thm}

\begin{proof} Interpret the latter form of the generating function from Theorem \ref{hfixsizek} as follows.

The constant term outside the sum is the shift given.

Let there be $s-1$ parts of size at least $k+1$, generated by $\frac{q^{(k+1)(s-1)}}{(q)_{s-1}}$. So we are allowed at most $s-1+h$ parts less than $k$.  Since the mex is $k$, there is at least one part each of sizes 1 through $k-1$, giving $q^{\binom{k}{2}}$.

Let as many as $s+h-k$ further parts of size at most $k-1$ be appended, generated by $\left[ {{s+h-1} \atop {k-1}} \right]_q$.
\end{proof}

The most direct connection to Andrews and Merca arises in the case $h=-1$, for now the $h+1$ term vanishes and the generating function becomes precisely $q^{-\binom{k}{2}} \mathcal{M}_k(q)$.  We have the following corollary:

\begin{cor}\label{MkCor} The quantity $M_k(n)$ can also be interpreted as the number of times in all partitions of $n-\binom{k}{2}$ that a $-1$-fixed hook arises from a part of size $k$.
\end{cor}

We also give the following combinatorial proof of Corollary \ref{MkCor}:

\begin{proof}[Bijective proof]
    Let $\lambda$ be a partition of $n$ with a $-1$-fixed hook in position $s$ arising from a part of size $k$. Thus $h_{1,s}(\lambda)=s-1$, $\lambda_s=k$, and $\lambda$ has $2s-k-1$ nonzero parts. Transform $\lambda$ into a new partition $\mu\vdash n+\binom{k}{2}$ as follows. Delete the part $\lambda_s=k$ and subtract one from the remaining $s-k-1$ nonzero parts $\lambda_{s+1}, \lambda_{s+2}, \ldots, \lambda_{2s-k-1}$ below it. Then add one to each of the first $s-1$ parts. Now append a part of every size from 1 to $k-1$. Thus $\mu$ is a partition of $n+\binom{k}{2}$, has minimal excludant $k$, and there are $s-1$ parts larger than $k$ and at most $s-2$ parts less than $k$ as desired. 
    
    This process is also invertible. Suppose $\mu \vdash N$ has mex $k$, $s-1$ parts greater than $k$ and at most $s-2$ parts less than $k$. First remove a single part of sizes $\{1, 2, \ldots, k-1$\} and insert a new part $\lambda_s=k$. Subtracting one from each of the first $s-1$ parts and adding one to parts $s, s+1, \ldots, 2s-k-1$ gives a partition of $N-\binom{k}{2}$ with a $-1$-fixed hook $h_{s,1}=s-1$ with $\lambda_s=k$ as desired,
    completing the proof.
\end{proof}

\begin{exm}
  Note that $\lambda = (11,6,5,5,4,4,\color{red}4\color{black},2,1) \vdash 42$ has a $-1$-fixed hook $h_{7,1}=6$  arising from a part of size $\lambda_7=4$. By taking the fixed hook from $\color{red}4$ and adding 1 to the previous parts we get $(12,7,6,6,5,5,1)$. After one adds on each part up to $4-1=3$ yields $\mu = (12,7,6,6,5,5,3,2,1,1) \vdash 48$.
\end{exm}

It is natural to wonder whether an interpretation similar to the pentagonal number recurrence could be given for the expressions arising from hook shifts other than $h=-1$.

\section{Fixing hook parameters}

The relevant data on the hook-length side of these identities appear to be the hook length $k$, the place $s$, and the fixedness $h$ of the hooks to be considered.  Individual terms indexed by $s$ in the generating functions above set particular values for the place.  If we set a value for the size of the hook instead of the part, we get the following generating function.

\begin{thm}\label{hfixhookk} The generating function for the number of partitions of $n$ with an $h$-fixed hook arising from a hook of size $k$ is $$\sum_{l=1}^k \frac{q^{k + l(k-h-1)}}{(q)_{k-h-1}} \left[{{k-1} \atop {l-1}} \right]_q.$$
\end{thm}

\begin{proof} The proof is direct from considering the contributions made by the various portions of the Young diagram of such a partition.

Let $l$ index the size of part that forms the hook of size $k$.

Since $k = s+h$ in place $s$, we have $s=k-h$ and so there are $k-h-1$ parts of size at least $l$, generated by $q^{l(k-h-1)}/(q)_{k-h-1}$.

There is a hook of size $k$, generated by $q^k$.

The leg of the hook is of length $k-l$, meaning that the $\rho$ portion of the partition is contributed by the $q$-binomial coefficient $\left[{{k-1} \atop {l-1}} \right]_q$.
\end{proof}

If we now sum over all hook sizes $k$, we get an expression for the generating function for all $h$-fixed hooks: 

\begin{thm}
    The generating function for the number of partitions of $n$ with an $h$-fixed hook is $$\sum_{k=1}^\infty \sum_{l=1}^k \frac{q^{k + l(k-h-1)}}{(q)_{k-h-1}} \left[{{k-1} \atop {l-1}} \right]_q  = \sum_{l=1}^\infty q^{l(-h-1)} \sum_{k=l}^\infty \frac{q^{(l+1)k}}{(q)_{k-h-1}} \left[{{k-1} \atop {l-1}} \right]_q.$$
\end{thm}

In the case $h=-1$, the inner sum becomes precisely $q^{-\binom{l}{2}} \mathcal{M}_l(q)$, and so we are getting a shifted sum over all of the $\mathcal{M}_l$.  Note that many sums involving fixed hook lengths strongly resemble Andrews and Merca's original $\mathcal{M}_k(q)$.  One wonders whether a deeper connection is being brought to light here, or if there is a more elegant combinatorial identity to be found, especially for all values of $h$.

If we instead hold $k$ and sum over all $h$, we get the generating function for the number of first-column $k$-hooks in all partitions of $n$.

\begin{thm}\label{firstkhooks}
    The generating function for the number of first column $k$-hooks in all partitions of $n$ is $$\frac{q^k}{(q^k;q)_\infty}\sum_{l=1}^k\frac{1}{(q)_{k-l}}.$$
\end{thm}

\begin{proof}
    The valid range for summation is $-\infty < h \leq k-1$.  Using the substitution $H=k-h-1$, the sum becomes

\begin{align*}\sum_{h=-\infty}^{k-1} \sum_{l=1}^k \frac{q^{k + l(k-h-1)}}{(q)_{k-h-1}} \left[{{k-1} \atop {l-1}} \right]_q &= q^k \sum_{H=0}^\infty \sum_{l=1}^k \frac{q^{lH}}{(q)_H} \left[ {{k-1} \atop {l-1}} \right]_q  \\
&= q^k \sum_{l=1}^k \left[ {{k-1} \atop {l-1}} \right]_q \sum_{H=0}^\infty \frac{q^{lH}}{(q)_H} \\ &= \sum_{l=1}^k \left[ {{k-1} \atop {l-1}} \right]_q \frac{1}{(q^l;q)_\infty} \\
&= \frac{q^k}{(q^k;q)_\infty} \sum_{l=1}^k \frac{1}{(q)_{k-l}}.
\end{align*}
\end{proof}

\bibliography{HLbibliography}

\end{document}